\documentclass[11pt]{article}
\usepackage[margin=1in]{geometry}
\usepackage{enumitem}
\usepackage{amsmath,amssymb,amsthm,nicefrac,hyperref}
\usepackage{thmtools}
\usepackage[svgnames]{xcolor}
\usepackage{listings}
\usepackage{setspace}
\usepackage{authblk}
\usepackage[utf8]{inputenc}
\newcommand{\QQ}{\mathbb{Q}}
\newcommand{\ZZ}{\mathbb{Z}}
\newcommand{\RR}{\mathbb{R}}

\newcommand{\CC}{\mathbb{C}}

\renewcommand{\phi}{\varphi}

\renewcommand{\epsilon}{\varepsilon}

\newcommand{\norm}[1]{\|{#1}\|}

\newcommand{\ip}[1]{\langle {#1} \rangle}

\newcommand{\diag}{\mathrm{diag}}
\newcommand{\spec}{\mathrm{spec}}
\newcommand{\tnorm}[1]{\big| {#1} \big |_{\mathbb{R}/\mathbb{Z}}}

\newtheorem{theorem}{Theorem}[section]
\newtheorem{lemma}[theorem]{Lemma}
\newtheorem{proposition}[theorem]{Proposition}

\newtheorem{conjecture}[theorem]{Conjecture}

\title{Uniform mixing and $\epsilon$-uniform mixing on cycles}
\author[a]{Xiwang Cao}
\author[b]{Cuiwen Zhu}
\affil[a]{Department of Mathematics, Nanjing University of Aeronautics and Astronautics, Nanjing 210016, China}
\affil[b]{Department of Combinatorics \& Optimization, University of Waterloo, Waterloo, ON N2L 3G1, Canada}
\date{July 1, 2026}

\begin{document}

\maketitle

\begin{abstract}
    We study continuous-time quantum walks on cycles. We prove two complementary results. Firstly, the cycle $C_9$ does not admit uniform mixing at any time. Using the similar idea and Dickson polynomials, we prove that $C_{15}$ does not admit uniform mixing at any time neither. Secondly, for every prime $p$, we show that the cycle $C_{p^2}$ admits $\epsilon$-uniform mixing.
\end{abstract}
\section{Introduction}
Quantum algorithms are drawing more researchers' attention nowadays. The development of efficient quantum algorithms has prompted the investigation of quantum walks, which were introduced by Farhi and Gutmann in 1998\cite{Farhi&Gutmann}. In this paper, we focus on the mixing properties of continuous-time quantum walks on cycles, especially in walks that reach uniform probability densities at particular times. We say such graphs admit \textit{uniform mixing}. More explicitly, let $X$ be a graph and $A$ be its adjacency matrix. Let $U(t)=\exp(itA)$ be the transition matrix. We say $X$ admits uniformly mixing at time $t$ if $U(t)$ is flat. Equivalently, since $U(t)$ is unitary (with inverse $U(-t)$), $X$ admits uniform mixing at $t$ if and only if
    $$|U(t)_{u,v}|^2=\frac{1}{n} \quad \text{for all $u,v \in V(X)$,}$$
    where $n$ is the number of vertices in $X$.\\
It was conjectured by Ahmadi, Belk, Tamon, and Wendler \cite{Ahmadi} that uniform mixing does not occur on the cycle of length $n$ if $n \ge 5$. In 2014, Godsil, Mullin, and Roy \cite{mullin17} showed that uniform mixing does not occur on $C_{2m}$ for $m \ge 3$ or $C_p$ for all primes $p \ge 5$. The next remaining case is $C_9$. In Theorem \ref{result1}, we prove that $C_9$ does not admit uniform mixing. Using a similar idea and the Dickson polynomials, we prove in Theorem \ref{result2} that $C_{15}$ does not admit uniform mixing.\\
We are also interested in which cycles come arbitrarily close to admitting uniform mixing. We say graphs of this property admit \textit{$\epsilon$-uniform mixing}. More explicitly, we say $X$ admits $\epsilon$-uniform mixing if and only if for any $\epsilon>0$, there exists some $t \in \RR$ such that
    $$\norm{U(t)\circ U(t)^*-\frac{1}{n}J}<\epsilon,$$
    where for any $A,B \in M_n(\CC)$,
    $$(A \circ B)_{k,l}=A_{k,l} \cdot B_{k,l},$$
    is the Schur product of matrices, and
    $$\norm{A-B}=\sqrt{\sum_{k=1}^n \sum_{l=1}^n |A_{k,l}-B_{k,l}|^2},$$
    is the Frobenius norm on matrices.\\
Godsil, Mullin, and Roy \cite{mullin17} showed that $C_{p}$ admits $\epsilon$-uniform mixing for all odd primes $p$. In this paper, we extend their work and prove that $C_{p^2}$ admits $\epsilon$-uniform mixing for all primes $p$ in Theorem \ref{result3}. We follow the mathematical notation given in \cite{chrisnotation}.

\section{Uniform Mixing}
Firstly, recall that the eigenvalues of a cycle $C_n$ have the form
$$\zeta_n^j+\zeta_n^{-j}, \quad j=0,\dots,n-1;$$
where $\zeta_n$ is the primitive $n$-th root of unity. Since we are working on cycles and therefore circulant matrices, we recall the following useful proposition.
\begin{proposition}\label{fourier}
    Let $C \in M_n(\CC)$ be a circulant matrix. Write
    $$C=\begin{pmatrix}
        c_0 & c_1 & \cdots & c_{n-1}\\
        c_{n-1} & c_0 & \cdots & c_{n-2}\\
        \vdots & \vdots & \ddots &\vdots\\
        c_1 & c_2 & \cdots & c_0
    \end{pmatrix}.$$
    Then the eigenvalues $\lambda_j$ of $C$ have the form
    $$\lambda_j=\sum_{k=0}^{n-1} c_k \zeta_n^{jk}.$$
    Dually, 
    $$c_j=\frac{1}{n}\sum_{k=0}^{n-1} \lambda_k \zeta_n^{-jk}.$$
\end{proposition}
Now, we give a characterization of $C_n$ admitting uniform mixing for general $n$. We first prove the following lemma.
\begin{lemma} \label{vandermonde}
    Let $n$ be a natural number that is greater than 2, $\zeta=\zeta_n$ be the primitive $n$-th root of unity. If
    $$\sum_{l=1}^{n-1} C_l \zeta^{-lm}=0, \quad \text{for all } m \in \{0,\dots, n-2\},$$
    then $C_l=0, \quad \text{for all } l=0,\dots,n-1$.
\end{lemma}
\begin{proof}
    The equations yield
    $$\begin{pmatrix}
        C_1&\cdots & C_{n-1} 
    \end{pmatrix}\underbrace{\begin{pmatrix}
        1 & \zeta^{-1} & \cdots & \zeta^{-(n-2)} \\
        \vdots &\vdots&\ddots &\vdots\\
        1 & \zeta^{-(n-1)} &\cdots & (\zeta^{-(n-1)})^{n-2}
    \end{pmatrix}}_{V}=(0,\dots,0).$$
    Notice $V$ is a Vandermonde matrix, and the $\zeta^{-j}$'s are distinct since we choose $\zeta$ to be a primitive root and therefore $\zeta^{-1}$ is also primitive. Hence $V$ is invertible, and we obtain
    $$(C_1,\dots,C_{n-1})=(0,\dots,0).$$
\end{proof}
\begin{theorem}\label{char}
    Let $\{\lambda_j=\zeta_n^j+\zeta_n^{-j}:j=0,\dots,n-1\}$ be the spectrum of $A=A(C_n)$. Let
    $$z_j:=e^{it\lambda_j}.$$
    Then $C_n$ admits uniform mixing if and only if 
    $$\sum_{k=0}^{n-1} \frac{z_{k+l}}{z_k}=0$$
    for any $l=1,\dots,n-1$.
\end{theorem}
\begin{proof}
    Notice that $A$ is circulant. So, if $f$ is analytic, $f(A)$ is also circulant. In particular, 
    $$U(t)=\exp(itA) \text{ is circulant.}$$
    Let $H:=\sqrt{n}U(t)$. Then $H$ is also circulant. Write the first row of $H$ to be
    $$\begin{pmatrix}
        c_0 & \cdots  & c_{n-1}
    \end{pmatrix}.$$
    $C_n$ admits uniform mixing if and only if $U(t)$ is flat with every entry having modulus $\frac{1}{\sqrt{n}}$, which is equivalent to $H$ is flat with every entry having modulus 1.\\
    Since $H=\sqrt{n}\exp(itA)$, and the eigenvalues for $A$ are $\lambda_j$'s, we know the eigenvalues for $H$ are
    $$\sqrt{n}e^{it\lambda_j}=\sqrt{n}z_j ,j=1,2,\dots,n.$$
    By Proposition \ref{fourier}, we have
    $$c_m=\frac{1}{n}\sum_{j=0}^{n-1} (\sqrt{n}z_j)\zeta_n^{-jm}=\frac{1}{\sqrt{n}} \sum_{j=0}^{n-1} z_j \zeta_n^{-jm}.$$
    Flatness of $H$ is equivalent to $|c_m|=1$ for all $m$, so
    \begin{align*}
        |c_m|^2&=\frac{1}{n}(\sum_{j=0}^{n-1} z_j \zeta_n^{-jm})(\sum_{k=0}^{n-1} \overline{z_k} \zeta_n^{km})\\
        &=\frac{1}{n} \sum_{j,k=0}^{n-1} z_j\overline{z_k} \zeta_n^{(k-j)m}\\
        &=\frac{1}{n} \sum_{l,k=0}^{n-1} z_{l+k}\overline{z_k} \zeta_n^{-lm}\\
        &=\frac{1}{n} \sum_{l=0}^{n-1} C_l \zeta_n^{-lm} \quad \text{, where }C_l=\sum_{k=0}^{n-1} z_{l+k}\overline{z_k}.
    \end{align*}
    Notice $z_k=e^{it\lambda_k}$ is unimodular, so
    $$C_0=\sum_{k=0}^{n-1} z_k\overline{z_k}=n.$$
    Therefore $|c_m|^2=1$ for all $m$ is equivalent to
    $$\sum_{l=1}^{n-1} C_l \zeta^{-lm}=0, \quad \text{for all } m \in \{0,\dots,n-1\}.$$
    In particular the equation holds for $m \in \{0,\dots,n-2\}$, so by Lemma \ref{vandermonde}, $C_l=0$, for all $l$ in $\{1,\dots,n-1\}$. Notice this satisfies the equation when $m=n-1$ as well. So we have
    $$\sum_{k=0}^{n-1} z_{l+k}\overline{z_k}=\sum_{k=0}^{n-1} \frac{z_{l+k}}{z_k}=0.$$
    for any $l=1,\dots, n-1$.
\end{proof}
\begin{theorem}\label{result1}
    $C_9$ does not admit uniform mixing.
\end{theorem}
\begin{proof}
    Let $A$ be the adjacency matrix of $C_9$, which is the cycle of nine vertices. Let $U(t)=e^{itA}$ be the transition matrix. Suppose, for contradiction, that there exists some $t$ such that $U(t)$ is flat and each entry has modulus $\frac{1}{\sqrt{9}}=\frac{1}{3}$. Let $\zeta=\zeta_9$. Then the eigenvalues of $A$ are
    $$\lambda_j=\zeta^j+\zeta^{-j},\quad j\in \{0,\dots,8\}$$
    and the eigenvalues of $U(t)$ are
    $$z_j=e^{it\lambda_j}, j\in \{0,\dots,8\}.$$
    Notice
    $$\lambda_0=2,\lambda_3=-1, \lambda_1+\lambda_2+\lambda_4=0, \lambda_j=\lambda_{9-j}, \text{ for all } j \in \{1,\dots,8\}.$$
    Let $$q=e^{-it}, r=e^{it\lambda_1},s=e^{it\lambda_2}.$$ 
    Then
    $$(z_0,\dots,z_8)=(q^{-2},r,s,q,(rs)^{-1},(rs)^{-1},q,s,r).$$
    By Theorem \ref{char}, $C_9$ admits uniform mixing if and only if
    \begin{equation}
        \sum_{k=0}^{8} \frac{z_{k+l}}{z_k}=0, \label{1}
    \end{equation}
    for any $l=1,\dots,n-1$.
    Let $$X=\frac{q}{r},\quad  Y=\frac{q}{s},\quad Z=qrs, \quad T(u)=u+u^{-1}.$$ 
    Expending (\ref{1}) for $l=1,\dots,4$ gives
    $$\begin{cases}
        1+T(YZ)+T(X/Y)+T(Y)+T(Z)=0,\\
        1+T(XZ)+T(X)+T(Y/Z)+T(Z)=0,\\
        1+T(XYZ)+T(X/Z)+T(Y/Z)+T(X/Y)=0,\\
        1+T(XY)+T(X/Z)+T(Y)+T(X)=0.
    \end{cases}$$
    Cleaning the denominators yields
    $$\begin{cases}
        P_1:=XYZ+XY^2Z^2+X+X^2Z+Y^2Z+XY^2Z+XZ+XYZ^2+XY=0,\\
        P_2:=XYZ+X^2YZ^2+Y+X^2YZ+YZ+XZ^2+XY^2+XYZ^2+XY=0,\\
        P_3:=XYZ+X^2Y^2Z^2+1+YZ^2+X^2Y+XZ^2+XY^2+X^2Z+Y^2Z=0,\\
        P_4:=XYZ+X^2Y^2Z+Z+YZ^2+X^2Y+XY^2Z+XZ+X^2YZ+YZ=0.
    \end{cases}$$
    Let $I=\ip{P_1,\dots,P_4}$. Then we obtain the Gr\"obner basis of I:
    $$\{X+Y+Z,Y^2+YZ+Z^2,Z^3-1\}.$$
    Therefore we can reduce the system of equations $\{P_i=0,i=1,\dots,4\}$ to
    \[
    \begin{aligned}
    X+Y+Z&=0,\\
    Y^2+YZ+Z^2&=0,\\
    Z^3-1&=0.
    \end{aligned}
    \]
    From $Z^3-1=0$, we see $Z \ne 0$, so $Y^2+YZ+Z^2=0$ divided by $Z^2$ yields to
    $$\left(\frac{Y}{Z}\right)^2+\frac{Y}{Z}+1=0.$$
    Hence $\frac{Y}{Z} \in \{\zeta_3,\zeta_3^2\}$. Say $\frac{Y}{Z}=\gamma$. Then $Y=\gamma Z$. Then $X+Y+Z=0$ gives
    $$X=-Y-Z=-(\gamma+1)Z=\gamma^2Z.$$
    Therefore
    $$X^3=(\gamma^2Z)^3=\gamma^6Z^3=1,$$
    $$Y^3=(\gamma Z)^3=\gamma^3 Z^3=1,$$
    $$XYZ=(\gamma^2Z)(\gamma Z)Z=\gamma^3Z^3=1.$$
    On the other hand,
    $$XYZ=\frac{q}{r}\cdot\frac{q}{s}\cdot qrs=q^3.$$
    Hence $q^3=1$. Recall that $q=e^{-it}$, so
    $$q^3=1 \implies t=\frac{2\pi m}{3}, m \in \ZZ.$$
    Now, $X=\frac{q}{r}$, so $X^3=1$ gives $r^3=q^3=1$. Recall that $r=e^{it\lambda_1}$. So
    $$1=r^3=e^{i3t\lambda_1}=e^{i2\pi(m\lambda_1)}$$
    Therefore $m\lambda_1 \in \ZZ$. Since $\lambda_1=2\cos(2\pi /9)$, it is irrational, therefore $m=0$. However, for $t=2\pi m/3=0$, $U(t)=I$ is not flat. Hence uniform mixing cannot occur at any time $t$.
\end{proof}

\section{Dickson polynomial}
The \textit{Dickson polynomials of the first kind} are defined by the recurrence relation for $k \ge 2$,
$$D_{k}(x) = xD_{k-1}(x)-D_{k-2}(x),$$
with the initial condition $D_0(x)=2, D_1(x)=x$. Then we have $D_k(x) \in \ZZ[x]$ for every $k \in \ZZ_{\ge 0}$.\\
The following proposition is well-known and can be proved directly by induction.
\begin{proposition}
    For every positive integer $k$, we have
    $$D_k(x+x^{-1})=x^k+x^{-k}.$$
\end{proposition}
Notice that the eigenvalues of cycle $C_n$ are
$$\lambda_j=\zeta_n^j+\zeta_n^{-j}=D_j(\lambda_1) \quad ,j=0,\dots,n-1,$$
where $\zeta_n$ is a primitive $n$-th root of unity. Hence we obtain a way to write $\lambda_j's$ in terms of $\lambda_1$. Since $[\mathbb{Q}(\lambda_1):\mathbb{Q}]=\frac{\phi(n)}{2}$, where $\phi$ is the Euler-phi function, and $\deg(D_j(x))=j$ with $D_j$ monic, we know that the algebraic extension $\QQ(\lambda_1)$ over $\mathbb{Q}$ has dimension $\frac{\phi(n)}{2}$ and has $\{1, \lambda_1,\lambda_2,\cdots,\lambda_{\frac{\phi(n)}{2}-1}\}$ as a basis. Thus every eigenvalue of $C_n$ is an integral linear combination of such a basis.
\begin{theorem} \label{result2}
    $C_{15}$ does not admit uniform mixing.
\end{theorem}
\begin{proof}
    Suppose, for contradiction, that $C_{15}$ admits uniform mixing at $t$. Let $$z_j:=e^{it\lambda_j}, j \in \{0,\dots,15\}$$
    be the eigenvalues of $U(t)=\exp(itA)$, where $A$ is the adjacency matrix of $C_{15}$ as above. Noticing that the minimal polynomial of $\lambda_1$ is
    $$p(x)=x^4-x^3-4x^2+4x+1.$$
    Denote
    $$x_0=e^{it},x_1=e^{it\lambda_1},x_2=e^{it(\lambda_1^2-2)},x_3=e^{it(\lambda_1^3-3\lambda_1)}.$$
    Then we obtain the following table ($z_{15-n}=z_n$ for any $n \in \{1,\dots,14\}$):
    \begin{center}
\begin{tabular}{|c|l|l|l|l|}
\hline
$n$ & $D_n(x) \pmod{p(x)}$ & $\lambda_n$ & $z_n$ & $\bar{z}_n$ \\ \hline
0 & $2$ & $2$ & $x_0^2$ & $x_0^{-2}$ \\
1 & $x$ & $\lambda_1$ & $x_1$ & $x_1^{-1}$ \\
2 & $x^2 - 2$ & $\lambda_1^2 - 2$ & $x_2$ & $x_2^{-1}$ \\
3 & $x^3 - 3x$ & $\lambda_1^3 - 3\lambda_1$ & $x_3$ & $x_3^{-1}$ \\
4 & $x^3 - 4x + 1$ & $\lambda_1^3 - 4\lambda_1 + 1$ & $x_0 x_1^{-1} x_3$ & $x_0^{-1} x_1 x_3^{-1}$ \\
5 & $-1$ & $-1$ & $x_0^{-1}$ & $x_0$ \\
6 & $-x^3 + 3x - 1$ & $-\lambda_1^3 + 3\lambda_1 - 1$ & $x_0^{-1} x_3^{-1}$ & $x_0 x_3$ \\
7 & $-x^3 - x^2 + 3x + 2$ & $-\lambda_1^3 - \lambda_1^2 + 3\lambda_1 + 2$ & $x_2^{-1} x_3^{-1}$ & $x_2 x_3$ \\ \hline
\end{tabular}
    \end{center}
    By Theorem \ref{char}, $C_{15}$ admits uniform mixing if and only if
    $$\sum_{k=0}^{14} \frac{z_{k+l}}{z_k}=\sum_{k=0}^{14}z_{k+l}\overline{z}_k=0$$
    for any $l=1,\dots,14$.
    Notice that for a fixed $l \in \{8,\dots,14\}$, we have
    $$0=\overline{\sum_{k=0}^{14} z_{k+l}\overline{z}_k}=\sum_{k=0}^{14} \overline{z}_{k+l}z_k=\sum_{j=0}^{14} z_{j-l}\overline{z}_{j}=\sum_{j=0}^{14}z_{j+(15-l)}\overline{z}_j.$$
    So it suffices to consider the equations with $l=1,\dots,7$. We obtain
    \begin{align*}
        P_1 : = {} & x^4zu^2 + x^4zu + x^3yu + x^3zu + x^2y^2u + x^2yz^2 \\
           & + x^2yzu^2 + x^2yzu + x^2yz + x^2yu^2 + x^2z^2u + xy^2zu + xyz^2u + y^2zu + y^2z = 0 \\
        P_2 : = {} & x^4yu^2 + x^4zu^4 + x^3yzu^3 + x^3yu^2 + x^3yu + x^3u^3 \\
           & + x^2y^2zu + x^2yzu^2 + x^2zu^3 + xy^2z^2u + xyz^2u^3 + xyz^2u^2 + xyzu + y^2z + yz^2u^2 = 0 \\
        P_3 : = {} & x^4y^2zu + x^3y^2z^2u^2 + x^3y^2zu^4 + x^3y^2u + x^3yz^2u^4 + x^3zu^3 
            + x^2y^3u^2 + x^2y^2zu^2\\& + x^2yz^2u^2 + xy^4zu + xy^3 + xy^2z^2u^3 
            + xy^2z + xy^2u^2 + y^2zu^3 = 0 \\
    \end{align*}
    \begin{align*}
P_4 : = {} & x^2y^2zu^2 + x^2y^2zu + x^2yz^2u^3 + x^2z^2u^4 + xy^2zu + xyz^2u^4 \\
           & + xyzu^3 + xyzu^2 + xyzu + xy + xzu^3 + y^2 + yu + zu^3 + zu^2 = 0 \\
P_5 : = {} & x^6y^2z^2u^2 + x^5yz^2u^4 + x^4y^3z^2u^3 + x^4z^2u^3 + x^3y^2z^4u^3  + x^3y^2z^3u^4 
+ x^3y^2z^3u \\&+ x^3y^2z^2u^2 + x^3y^2zu^3 + x^3y^2z + x^3y^2u + x^2y^4z^2u + x^2yz^2u + xy^3z^2 + y^2z^2u^2 = 0 \\
P_6 : = {} & x^6yz^2u^3 + x^5z^2u^3 + x^4y^2z^2u^2 + x^4yz^2u^4 + x^4zu^3 \\
           & + x^3y^2z^3u^3 + x^3yz^4u^3 + x^3yz^2u^2 + x^3yu + x^3zu + x^2y^2z^3u \\
           & + x^2yz^2 + x^2z^2u^2 + xy^2z^2u + yz^2u = 0 \\
P_7 : = {} & x^4yz^2u^2 + x^3y^2zu^2 + x^3yz^2u^2 + x^3yz^2u + x^3yzu^2 \\
           & + x^3zu + x^2y^2z^2u^2 + x^2yzu + x^2 + xy^2zu + xyz \\
           & + xyu + xy + xz + y = 0,
\end{align*}
where $x=x_0,y=x_1,z=x_2,u=x_3$. \\ Let $I=\ip{P_1,\dots,P_7}$. Then the Gr\"obner basis of $I$ is
\begin{align*}
     B_1 := {}&x^2 + xz - y^2u^2 - 2y^2u - y^2 + 3yzu^3 + 2yzu^2 - yzu +2yu^5 \\& +yu^4- yu^3 - yu^2 - yu + y+ z^3u^3 + z^3u^2- z^2u^7
    + z^2u^6 + 4z^2u^5 \\&+ 4z^2u^4 + z^2u^3 - z^2u^2 + 2zu^7 +3zu^6 - 2zu^4 - 2zu^3 - zu^2,\\
    B_2:= {} &
    xy + y^2u + y^2 - 2yzu^3 - 2yzu^2 - 3yu^5 - yu^4 + yu^3 + yu + 9z^2u^7 + 11z^2u^6 \\
    &- 4z^2u^4 - 2z^2u^3 - 3zu^7 - 4zu^6 + zu^4 + zu^3 + zu^2,\\
    B_3 := {} & xz^3u + y^2u^2 - yzu^3 - yzu^2 - 2z^3u^3 - 2z^3u^2 - 6z^2u^7 - 11z^2u^6 - 5z^2u^5 - z^2u^4 - z^2u^3,\\
    B_4 := {} & xzu^2 - y^2u^2 - y^2 - yzu^3 + yzu - 6yu^5 - 3yu^4 + yu^3 - yu^2 + \\
    &\frac{23}{2}z^2u^7 + 11z^2u^6 - \frac{1}{2}z^2u^5 + z^2u^3 + z^2u^2 - 6zu^7 - 9zu^6 - 2zu^5 - zu^3,\\
    B_5 := {} &  y^3 + y^2u^2 + 2z^2u^7 + 3z^2u^6 - z^2u^4,\\
    B_6 := {} & y^2z + y^2u^2 - yzu^3 - yzu^2 + yu^5 + yu^4 + 5z^2u^7 + 10z^2u^6  \\
    &+ 3z^2u^5 - 3z^2u^4 - z^2u^3 + zu^7 + 2zu^6 + zu^5,\\
    B_7:= {} &y^2u^3 - 2z^2u^7 - 4z^2u^6 - 2z^2u^5,\\
    B_8 := {} &yz^2u + yzu^3 + yzu^2 + z^3u^3 + z^3u^2 + 5z^2u^7 + 9z^2u^6 +  4z^2u^5 + z^2u^4 + z^2u^3,\\
    B_9 := {} &yzu^4 + yzu^3 - 3/2z^2u^7 - 2z^2u^6 + 1/2z^2u^5 + z^2u^4,\\
    B_{10} := {}&yu^6 + yu^5 + z^2u^7 + z^2u^6 + zu^8 + 2zu^7 + zu^6,\\
    B_{11} := {}&z^4u^3 + z^4u^2 - z^2u^7 - 2z^2u^6 - z^2u^5,\\
    B_{12} := {}&z^3u^4 + z^3u^3,\\
    B_{13} := {}&z^2u^8 + 2z^2u^7 + z^2u^6.
\end{align*}
Hence solving the system of $P_j=0, j\in\{1,\dots,7\}$ is equivalent to solving the system of $B_k,k \in \{1,\dots,13\}$. From $B_{12}=0$, we obtain $u=-1$. Plugging $u=-1$ back to $B_{7}=0$, we obtain $y=0$, contradicts to $y=e^{it\lambda_1}$ is unimodulus.
\end{proof}
\section{Spectra of $C_{p^m}$}
We first analyze the relationships between the eigenvalues of the adjacency matrix of $C_{p^m}$.
\begin{lemma} \label{spec}
    Let $p$ be an odd prime, $m$ be a positive integer, $\zeta$ be the primitive $p^m$-th root of unity, and $A=A(C_{p^m})$. Then $\spec(A)=\{\lambda_j:=\zeta^j+\zeta^{-j}, j=0,\dots,p^m-1\}$. Let
    $$B=\{kp^{m-1}:1 \le k \le \frac{p-1}{2}\} \cup \{a+kp^{m-1}:1 \le a \le \frac{p^{m-1}-1}{2}, 0 \le k \le p-2\}.$$
    Then $\{\lambda_b\}_{b \in B}$ is $\QQ$-linearly independent, and $\spec(A) \subseteq\text{span}_{\ZZ}\{\lambda_b:b \in B\}$.
\end{lemma}
\begin{proof}
    We first show $\{\lambda_b:b \in B\}$ is $\QQ$-linearly independent. Suppose, for contradiction, that there exists some $q_b \in \ZZ$ not all zero such that 
    $$\sum_{b \in B}q_b\lambda_b=0.$$
    Then we obtain
    $$\sum_{b \in B}\left(q_b\zeta^{b}+q_b\zeta^{p^m-b}\right)=0.$$
    Let $$P(x)=\sum_{b \in B} q_b(x^b+x^{p^m-b}).$$ Then $P(x)$ is a polynomial with integral coefficients and has degree at most $p^m-1$, as $0 \not \in B$, and $P(\zeta)=0$. Notice the minimal polynomial of $\zeta$ over $\QQ$ is
    $$\Phi_{p^m}(x)=x^{(p-1)p^{m-1}}+\dots+x^{p^{m-1}}+1.$$
    Therefore $\Phi_{p^m}|P$. Since $\deg(\Phi_{p^m})=(p-1)p^{m-1}$, we have
    $$P(x)=Q(x)\Phi_{p^m}(x),$$
    where $Q(x)\in \ZZ[x]$ has degree at most $(p^m-1)-((p-1)p^{m-1})=p^{m-1}-1$. Since $\deg(Q)<p^{m-1}$, and the degree gaps in $\Phi_{p^m}(x)$ are $p^{m-1}$, we have
    $$[x^{a}]P=\dots=[x^{a+(p-1)p^{m-1}}]P, \quad \text{for all $a=1,\dots,\frac{p^{m-1}-1}{2}$},$$
    and
    $$[x^{0}]P=\dots=[x^{\frac{p-1}{2}p^{m-1}}]P,$$
    where $[x^j]P$ denotes the coefficient of $x^j$ in $P(x)$. However, since $a+(p-1)p^{m-1} \not \in B$ and $a+(p-1)p^{m-1} \not \in p^m-B$ for any $1\le a \le \frac{p^{m-1}-1}{2}$, we conclude that
    $$[x^{a}]P=\dots=[x^{a+(p-1)p^{m-1}}]P=0, \quad \text{for all $a=1,\dots,\frac{p^{m-1}-1}{2}$}.$$
    Also, since $0 \not \in B$ and $0 \not \in p^m-B$, we conclude that
    $$0=[x^{0}]P=\dots=[x^{\frac{p-1}{2}p^{m-1}}]P.$$
    Thus, $q_b=0$ for all $b \in B$, contradicts to $q_b$'s not all zero.\\
    Since $[\mathbb{Q}(\lambda_1):\mathbb{Q}]=\frac{p^{m-1}(p-1)}{2}=|B|$. $\{\lambda_b: b\in B\}$ contributes a basis for $\mathbb{Q}(\lambda_1)$. In fact, we can provide the explicit formula of every eigenvalue of $A$ as an integral linear combination of $\{\lambda_b: b\in B\}$ as follows.
    We first notice that for any $a=0,\dots,p^{m-1}-1$,
    \begin{align*}
        \sum_{k=0}^{p-1} \lambda_{a+kp^{m-1}} &= \sum_{k=0}^{p-1} \zeta^{a+kp^{m-1}}+\zeta^{-a-kp^{m-1}}\\
        &=\zeta^a \sum_{k=0}^{p-1}\zeta^{kp^{m-1}}+\zeta^{-a}\sum_{k=0}^{p-1}\zeta^{-kp^{m-1}}\\
        &=\zeta^a \cdot 0 + \zeta^{-a} \cdot 0\\
        &=0.
    \end{align*}
    Let $a \in \{1,\dots,\frac{p^{m-1}-1}{2}\}$ be fixed. Then 
    \begin{equation*}
        \lambda_{a+(p-1)p^{m-1}}=-\sum_{k=0}^{p-2} \lambda_{a+kp^{m-1}} \in \mathrm{span}_{\ZZ}\{\lambda_b:b \in B\}.
    \end{equation*}
    Now, let $k \in \{0,\dots,p-1\}$ be fixed, we have
    $$\lambda_{(p^{m-1}-a)+(p-k-1)p^{m-1}}=\lambda_{p^m-(a+kp^{m-1})}=\lambda_{a+kp^{m-1}} \in \mathrm{span}_{\ZZ}\{\lambda_b:b \in B\}.$$
    Notice $p^{m-1}-a \in \{\frac{p^{m-1}+1}{2},\dots,p^{m-1}-1\}$ and $p-k-1 \in \{0,\dots, p-1\}$. Therefore
    $$\{\lambda_{a+kp^{m-1}}:1\le a \le p^{m-1}-1,0\le k\le p-1\} \subseteq \mathrm{span}_{\ZZ}\{\lambda_b:b \in B\}$$
    Also, for any $k \in \{1,\dots, \frac{p-1}{2}\}$, we have
    $$\lambda_{kp^{m-1}}=\lambda_{p^m-kp^{m-1}}=\lambda_{(p-k)p^{m-1}}.$$
    Since $k \in \{1,\dots, \frac{p-1}{2}\}$, we have $p-k \in \{\frac{p+1}{2},\dots,p-1\}$. Hence
    $$\lambda_{pk} \in \{\lambda_b:b \in B\} \quad \text{, for all $k=1,\dots,p-1$.}$$
    Moreover,
    $$\sum_{k=0}^{p-1}\lambda_{kp^{m-1}}=0,$$
    we have
    $$\lambda_0 = -\sum_{k=1}^{p-1}\lambda_{kp^{m-1}} \in \mathrm{span}_{\ZZ} \{\lambda_b:b \in B\}.$$
    Thus, we conclude that $\spec(A) \subseteq \mathrm{span}_{\ZZ}\{\lambda_b:b \in B\}$.
\end{proof}
The rest of the proof relies heavily on Kronecker's theorem. Here we follow the version in \cite{kronecker}. We say $t$ is a \textit{(topological) generator} of a compact torus $T$ if $t$ is an element of $T$, and the smallest closed subgroup of $T$ containing $t$ is $T$ itself.
\begin{theorem}[Kronecker's Theorem]
    Let $(t_1,\dots,t_r)$ denote an element of $\RR^r$, and let $t$ be the image of this point in $T=(\RR/\ZZ)^r$. Then $t$ is a generator of $T$ if and only if $\{1,t_1,\dots,t_r\}$ are linearly independent over $\QQ$.
\end{theorem}

\section{\texorpdfstring{$\epsilon$}{epsilon}-Uniform Mixing}
Let $A=A(C_{p^2})$ be the adjacency matrix of $C_{p^2}$, where $p$ is an odd prime. Then $A$ is circulant, and the eigenvalues of $A$ are $\{\lambda_j=\zeta^j+\zeta^{-j},j=0,\dots,p^2-1\}$, where $\zeta$ is the primitive $p^2$-th root of unity. Define $F$ to be a $p^2 \times p^2$ complex matrix with
$$F_{j,k}=\frac{1}{p}\zeta^{jk}.$$
Then the $j$-th column of $F$ is an eigenvector corresponding to $\lambda_j$. So we obtain
$$A=F \diag(\lambda_j)F^*.$$
Therefore
$$U(t)=\exp(itA)=F\diag(e^{it\lambda_j})F^*.$$
Now, for any $j \in \ZZ_{p^2}$, there exists unique $r,l \in \{0,\dots,p-1\}$ such that $j\equiv r+pl \pmod{p^2}$. Let $h \in \{0,\dots,p-1\}$ be such that $h \equiv 2^{-1} \pmod{p}$. Define
\begin{equation}
    \eta_j:=\zeta^{pr(l+h)}, \label{2}
\end{equation}
and
$$H:=F \diag(\eta_j)F^*.$$
\begin{proposition} \label{eta}
    For $\eta_j$ defined above, we have:
    \begin{enumerate}
        \item for any $j \in \ZZ_{p^2}$,
        $$\eta_j=\eta_{-j};$$
        \item for any $r \in \ZZ_p$,
        $$\prod_{l\in \ZZ_p} \eta_{r+pl}=1.$$
    \end{enumerate}
\end{proposition}
\begin{proof}
    \begin{enumerate}
        \item Write $j=r+pl$ with $r,l \in \ZZ_p$. Then 
        $$-j=-r-pl \equiv (p-r)+p(p-l-1) \pmod{p^2}$$
        Hence by definition (\ref{2}),
        $$\eta_{-j} = \zeta^{p(p-r)(p-l-1+h)}
            =\zeta^{pr(l+1-h)}.$$
        Since $h\equiv 2^{-1} \pmod{p}$ by construction, we have
        $$l+1-h \equiv l+h \pmod{p}.$$
        Hence
        $$pr(l+1-h) \equiv pr(l+h) \pmod{p^2}.$$
        Therefore,
        $$\eta_{-j}=\zeta^{pr(l+1-h)}=\zeta^{pr(l+h)}=\eta_{j}.$$
        \item Let $r \in \{0,\dots,p-1\}$ be fixed. Then
        $$\prod_{l=0}^{p-1} \eta_{r+pl} =\prod_{l=0}^{p-1} \zeta^{pr(l+h)}=\zeta^{\sum_{l=0}^{p-1} pr(l+h)}.$$
        Notice
        $$\sum_{l=0}^{p-1}pr(l+h)=p^2rh+pr\sum_{l=0}^{p-1}l=p^2r(h+\frac{p-1}{2}) \equiv 0 \pmod{p^2}.$$
        Since $\zeta$ is the primitive $p^2$-th root of unity, we have
        $$\prod_{l=0}^{p-1}\eta_{r+pl}=1.$$
    \end{enumerate}
\end{proof}
\begin{lemma} \label{flat-H}
    The matrix $H$ is a flat unitary matrix, and $|H_{x,y}|=\frac{1}{p}$ for all $x,y \in \ZZ_{p^2}$.
\end{lemma}
\begin{proof}
    We first notice that $F$ is unitary, since $F$ is a discrete Fourier transform matrix. Also, since $\eta_j=\zeta^{pr(l+h)}$, we have each $\eta_j$ is unimodular. Therefore
    \begin{align*}
        HH^*&=(F\diag(\eta_j)F^*)(F\diag(\eta_j)F^*)^*\\
        &=F\diag(\eta_j)F^*F\diag(\overline{\eta_j})F^*\\
        &=I.
    \end{align*}
    Similarly, $HH^*=I$ and hence $H$ is unitary.\\
    Notice that
    \begin{align*}
        H_{x,y} &= \sum_{j=0}^{p^2-1}F_{x,j} \eta_j (F^*)_{j,y}\\
        &=\frac{1}{p^2}\sum_{j=0}^{p^2-1} \eta_j \zeta^{(x-y)j}\\
        &=\frac{1}{p^2} \sum_{r,l=0}^{p-1} \zeta^{pr(l+h)+(x-y)(r+pl)}\\
        &=\frac{1}{p^2}\sum_{r=0}^{p-1} \zeta^{r(x-y)+prh}\big(\sum_{l=0}^{p-1} \zeta^{pl(r+(x-y))}\big).
    \end{align*}
    Since
    $$\sum_{l=0}^{p-1} \zeta^{pl(r+(x-y))}=\begin{cases}
        0 &\text{, if $r+(x-y) \not \equiv 0 \pmod{p}$,}\\
        p &\text{, if $r+(x-y)\equiv0 \pmod{p}$}.
    \end{cases}$$
    So
    $$H_{x,y} =\frac{1}{p} \zeta^{-ph(x-y)-(x-y)^2}.$$
    In particular,
    $$|H_{x,y}|=\frac{1}{p}=\sqrt{\frac{1}{p^2}},$$
    as desired.
\end{proof}
We cite Lemma 4.4.3 in \cite{natalie-thesis} due to Mullin.
\begin{lemma}
    Suppose that $A$ and $B$ are symmetric $n \times n$ complex matrices, such that
    $$\norm{A-B} \le \epsilon,$$
    for some positive real $\epsilon$. Then
    $$\norm{A \circ A^*-B\circ B^*}\le 2\epsilon.$$
\end{lemma}
Finally, we prove the result that $C_{p^2}$ admits $\epsilon$-uniform mixing.
\begin{theorem} \label{result3}
    Let $p$ be an arbitrary prime and $C_{p^2}$ be the cycle of length $p^2$. Then $C_{p^2}$ admits $\epsilon$-uniform mixing.
\end{theorem}
\begin{proof}
    Firstly, if $p=2$, Ahmadi et al. showed that $C_4$ admits uniform mixing \cite{Ahmadi}, and hence admits $\epsilon$-uniform mixing. Now we may assume $p$ is an odd prime.\\
    Let $A=A(C_{p^2})$ and $H$ be defined as above. The goal is to show $e^{it\lambda_j}$ gets arbitrarily close to $\eta_j$ for any $j$ at some time $t$. By Lemma 0.3, $\{\lambda_b:b \in B\}$ is linearly independent over $\QQ$, we now claim $$\{1\} \cup \{\frac{\lambda_b}{2\pi}:b \in B\}$$
    is linearly independent over $\QQ$.\\
    \textit{proof of the claim:} Suppose that there exists rational $q_b:b \in B \cup \{0\}$ such that
        $$q_0 \cdot 1 + \sum_{b \in B}q_b \frac{\lambda_b}{2\pi}=0.$$
        Therefore
        $$-\sum_{b \in B} q_b \lambda_b = q_0\cdot 2\pi.$$
        Notice that $\lambda_b$'s are algebraic, hence
        $$-\sum_{b \in B}q_b\lambda_b$$
        is algebraic. Since $\pi$ is transcendental and $q_0$ is rational, we conclude $q_0=0$. Thus
        $$\sum_{b \in B}q_b\lambda_b=0$$
        Linear independence yields $q_b=0$, for all $b \in B$. This finishes up the proof of claim.\\
    Now, for each $b \in B$, we can find $\beta_b \in \RR/\ZZ$ such that
    $$\eta_b = e^{2\pi i\beta_b}.$$
    By Kronecker's theorem, we see that
    $$D=\bigg\{\left(\frac{\lambda_b}{2\pi}m\right)_{b \in B}: m \in \ZZ  \bigg\}$$
    is dense in $(\RR/\ZZ)^{|B|}$. Hence for any $\delta>0$, we can find some $m \in \ZZ$ such that for every $b \in B$,
    $$\big| m\frac{\lambda_b}{2\pi}-\beta_b \big|_{\RR/\ZZ}<\delta.$$
    We choose $t=m$. We aim to bound
    \begin{equation}
        \big|m\frac{\lambda_j}{2\pi}-\beta_j \big|_{\RR/\ZZ}.
    \end{equation}
    Notice that as the proof in Lemma \ref{spec}, for all $1 \le a \le p-1, 0\le k \le p-1$, we have
    $$\lambda_{a+kp}=\lambda_{p^2-(a+kp)}, \quad \lambda_{kp}=\lambda_{p^2-kp}.$$
    By Proposition \ref{eta}, for any $0 \le j \le p^2$, we have
    $$\eta_j=\eta_{p^2-j}.$$
    Thus, it suffices to bound (4) for
    $$j \in \{a+(p-1)p:1\le a \le \frac{p-1}{2}\} \cup \{0\}.$$
    When $j=0$, we notice that
    $$\eta_{kp}=\zeta^{0}=1 \quad \text{, for all $0\le k\le p-1$}$$
    Therefore we have
    $$\beta_{kp}=0 \quad \text{, for all $0\le k\le p-1$}$$
    Recall we have
    $$\lambda_0=-\sum_{k=1}^{p-1} \lambda_{pk}.$$
    Then
    \begin{align*}
        \tnorm{m\frac{\lambda_{0}}{2\pi}-\beta_0} &=\tnorm{-m\frac{\sum_{k=1}^{p-1} \lambda_{pk}}{2\pi}}\\
        &\le \sum_{k=1}^{p-1} \tnorm{m\frac{\lambda_{pk}}{2\pi}}\\
        &=\sum_{k=1}^{p-1} \tnorm{m\frac{\lambda_{pk}}{2\pi}-\beta_{pk}}\\
        &<(p-1)\delta.
    \end{align*}
    Now we bound the rest. Before bounding it, we first claim that for any $a \in \{1,\dots, p-1\}$,
    $$\sum_{k=0}^{p-1} \beta_{a+kp} \in \ZZ.$$
    \textit{proof of the claim:} Let $a \in \{1,\dots, p-1\}$ be fixed. Then we have
        $$\prod_{k=0}^{p-1}\eta_{a+kp}=1.$$
        Simplifying it with $\eta_{a+kp}=e^{2\pi i \beta_{a+kp}}$, we obtain
        $$e^{2\pi i \sum_{k=0}^{p-1} \beta_{a+kp}}=1.$$
        Therefore
        $$\sum_{k=0}^{p-1} \beta_{a+kp} \in \ZZ.$$
        Which finishes up the proof of the claim.\\
    We now have
    \begin{align*}
        \tnorm{m\frac{\lambda_{a+(p-1)
        p}}{2\pi}-\beta_{a+(p-1)p}} &= \tnorm{\sum_{k=0}^{p-2} \beta_{a+kp}-\sum_{k=0}^{p-2}m\frac{\lambda_{a+kp}}{2\pi}}\\
        &\le \sum_{k=0}^{p-2} \tnorm{m\frac{\lambda_{a+kp}}{2\pi}-\beta_{a+kp}}\\
        &<(p-1)\delta.
    \end{align*}
    Notice that if $\tnorm{x-y}<\gamma$, then $|e^{2\pi ix}-e^{2\pi iy}|<2\pi \gamma$. Hence
    $$\sum_{j=0}^{p^2-1} |e^{it\lambda_j}-\eta_j| <(p^2-p)\cdot 2\pi \delta+p\cdot (p-1)2\pi \delta=4(p^2-p)\pi\delta.$$
    Let $\epsilon>0$ be fixed and choose $\delta$ such that
    $$4(p^2-p)\pi\delta<\frac{\epsilon}{2}.$$
    Then
    \begin{align*}
        \sum_{x,y \in \ZZ_{p^2}} |(U(t)-H)_{x,y}|^2 &= \sum_{x,y \in \ZZ_{p^2}} \left\vert \frac{1}{p^2}\sum_{j=0}^{p^2-1}(e^{it\lambda_j}-\eta_j) \zeta^{(x-y)j}\right\vert ^2\\
        &\le \sum_{x,y \in \ZZ_{p^2}} \left(\frac{1}{p^2}\sum_{j=0}^{p^2-1}|e^{it\lambda_j}-\eta_j|\right)^2\\
        &<\sum_{x,y \in \ZZ_{p^2}} \frac{1}{p^4} (4(p^2-p)\pi\delta)^2\\
        &<(\frac{\epsilon}{2})^2.
    \end{align*}
    Therefore,
    $$\norm{U(t)-H}<\frac{\epsilon}{2}.$$
    By Lemma \ref{flat-H}, $|H_{x,y}|=\frac{1}{p}$ and $H$ is unitary. Also, $H$ is symmetric by definition. Hence
    $$\norm{U(t)\circ U(t)^*-H \circ H^*} < \epsilon.$$
    Since
    $$H \circ H^*=\frac{1}{p^2}J.$$
    We conclude that $C_{p^2}$ admits $\epsilon$-uniform mixing.
\end{proof}
\section{Future Problems}
Firstly, our results support the conjecture of Ahmadi et al. that no cycle $C_n$ other than $C_3$ and $C_4$ admits uniform mixing. It is desirable for us to push the result of $C_9$ does not admit uniform mixing to $C_{p^m}$. More generally, it might be feasible to use the characterization of $C_n$ admitting uniform mixing to show that if $C_{pq}$ admits uniform mixing, then either $p$ or $q$ admits uniform mixing.\\
Secondly, as in the proof of $C_9$ does not admit uniform mixing, we see that the potential mixing time $t$ gives $e^{it}$ a third root of unity. Mullin has a conjecture\cite{mullin17}:
\begin{conjecture}
    For any graph $X$, if it admits uniform mixing at $t$, then $e^{it}$ is a root of unity.
\end{conjecture}
\noindent Lastly, we already prove that $C_{p^2}$ admits $\epsilon$-uniform mixing for all primes $p$. However, the construction of $\eta_j$'s fails for $C_{p^m}$ in general when $m \ge 3$. We feel obliged to look at $\epsilon$-uniform mixing on $C_{27}$ before making any conjectures.
\section{Acknowledgement}
The authors would like to express sincere gratitude to Chris Godsil for his invaluable guidance, encouragement, and support throughout this work. The second author is also grateful to Ada Chan for many helpful discussions, thoughtful suggestions, and generous advice during the preparation of this paper.

\clearpage
\appendix

\section{Gr\"obner basis calculation}
For $C_9$, we use Sage to compute the Gr\"obner basis.
\begin{lstlisting}[language=Python, caption={Code used for the computation $C_9$.}]
R.<X,Y,Z> = PolynomialRing(QQ, order='lex')
P1 = X*Y*Z + X*Y^2*Z^2 + X + X^2*Z + Y^2*Z \
+ X*Y^2*Z + X*Z + X*Y*Z^2 + X*Y
P2 = X*Y*Z + X^2*Y*Z^2 + Y + X^2*Y*Z + Y*Z \
+ X*Z^2 + X*Y^2 + X*Y*Z^2 + X*Y
P3 = X*Y*Z + X^2*Y^2*Z^2 + 1 + Y*Z^2 + X^2*Y \
+ X*Z^2 + X*Y^2 + X^2*Z + Y^2*Z
P4 = X*Y*Z + X^2*Y^2*Z + Z + Y*Z^2 + X^2*Y \
+ X*Y^2*Z + X*Z + X^2*Y*Z + Y*Z
I = ideal([P1,P2,P3,P4])
I.groebner_basis()
\end{lstlisting}
The output is
\begin{lstlisting}
[X + Y + Z, Y^2 + Y*Z + Z^2, Z^3 - 1]
\end{lstlisting}
For $C_{15}$, we use MAGMA to compute the Gr\"obner basis.
\begin{lstlisting}[language=Python, caption={Code used for the computation $C_{15}$.}]
K := RationalField();

F<x,y,z,u> := PolynomialRing(K, 4, "lex");

P1 := x^4*z*u^2 + x^4*z*u + x^3*y*u + x^3*z*u
    + x^2*y^2*u + x^2*y*z^2 + x^2*y*z*u^2 + x^2*y*z*u
    + x^2*y*z + x^2*y*u^2 + x^2*z^2*u + x*y^2*z*u
    + x*y*z^2*u + y^2*z*u + y^2*z;

P2 := x^4*y*u^2 + x^4*z*u^4 + x^3*y*z*u^3 + x^3*y*u^2
    + x^3*y*u + x^3*u^3 + x^2*y^2*z*u + x^2*y*z*u^2
    + x^2*z*u^3 + x*y^2*z^2*u + x*y*z^2*u^3
    + x*y*z^2*u^2 + x*y*z*u + y^2*z
    + y*z^2*u^2;

P3 := x^4*y^2*z*u + x^3*y^2*z^2*u^2 + x^3*y^2*z*u^4
    + x^3*y^2*u + x^3*y*z^2*u^4 + x^3*z*u^3
    + x^2*y^3*u^2 + x^2*y^2*z*u^2 + x^2*y*z^2*u^2
    + x*y^4*z*u + x*y^3 + x*y^2*z^2*u^3 + x*y^2*z
    + x*y^2*u^2 + y^2*z*u^3;

P4 := x^2*y^2*z*u^2 + x^2*y^2*z*u + x^2*y*z^2*u^3
    + x^2*z^2*u^4 + x*y^2*z*u + x*y*z^2*u^4
    + x*y*z*u^3 + x*y*z*u^2 + x*y*z*u + x*y
    + x*z*u^3 + y^2 + y*u + z*u^3 + z*u^2;

P5 := x^6*y^2*z^2*u^2 + x^5*y*z^2*u^4
    + x^4*y^3*z^2*u^3 + x^4*z^2*u^3
    + x^3*y^2*z^4*u^3 + x^3*y^2*z^3*u^4
    + x^3*y^2*z^3*u + x^3*y^2*z^2*u^2
    + x^3*y^2*z*u^3 + x^3*y^2*z + x^3*y^2*u
    + x^2*y^4*z^2*u + x^2*y*z^2*u + x*y^3*z^2
    + y^2*z^2*u^2;

P6 := x^6*y*z^2*u^3 + x^5*z^2*u^3
    + x^4*y^2*z^2*u^2 + x^4*y*z^2*u^4
    + x^4*z*u^3 + x^3*y^2*z^3*u^3
    + x^3*y*z^4*u^3 + x^3*y*z^2*u^2
    + x^3*y*u + x^3*z*u + x^2*y^2*z^3*u
    + x^2*y*z^2 + x^2*z^2*u^2
    + x*y^2*z^2*u + y*z^2*u;

P7 := x^4*y*z^2*u^2 + x^3*y^2*z*u^2
    + x^3*y*z^2*u^2 + x^3*y*z^2*u
    + x^3*y*z*u^2 + x^3*z*u
    + x^2*y^2*z^2*u^2 + x^2*y*z*u + x^2
    + x*y^2*z*u + x*y*z + x*y*u + x*y + x*z + y;

B := [P1, P2, P3, P4, P5, P6, P7];

I := ideal<F | B>;

G := GroebnerBasis(I);

G;
\end{lstlisting}
The output is
\begin{lstlisting}
    [
    x^2 + x*z - y^2*u^2 - 2*y^2*u - y^2 + 3*y*z*u^3 + 2*y*z*u^2 - y*z*u +
        2*y*u^5 + y*u^4 - y*u^3 - y*u^2 - y*u + y + z^3*u^3 + z^3*u^2 - z^2*u^7
        + z^2*u^6 + 4*z^2*u^5 + 4*z^2*u^4 + z^2*u^3 - z^2*u^2 + 2*z*u^7 +
        3*z*u^6 - 2*z*u^4 - 2*z*u^3 - z*u^2,
    x*y + y^2*u + y^2 - 2*y*z*u^3 - 2*y*z*u^2 - 3*y*u^5 - y*u^4 + y*u^3 + y*u +
        9*z^2*u^7 + 11*z^2*u^6 - 4*z^2*u^4 - 2*z^2*u^3 - 3*z*u^7 - 4*z*u^6 +
        z*u^4 + z*u^3 + z*u^2,
    x*z^3*u + y^2*u^2 - y*z*u^3 - y*z*u^2 - 2*z^3*u^3 - 2*z^3*u^2 - 6*z^2*u^7 -
        11*z^2*u^6 - 5*z^2*u^5 - z^2*u^4 - z^2*u^3,
    x*z*u^2 - y^2*u^2 - y^2 - y*z*u^3 + y*z*u - 6*y*u^5 - 3*y*u^4 + y*u^3 -
        y*u^2 + 23/2*z^2*u^7 + 11*z^2*u^6 - 1/2*z^2*u^5 + z^2*u^3 + z^2*u^2 -
        6*z*u^7 - 9*z*u^6 - 2*z*u^5 - z*u^3,
    y^3 + y^2*u^2 + 2*z^2*u^7 + 3*z^2*u^6 - z^2*u^4,
    y^2*z + y^2*u^2 - y*z*u^3 - y*z*u^2 + y*u^5 + y*u^4 + 5*z^2*u^7 + 10*z^2*u^6
        + 3*z^2*u^5 - 3*z^2*u^4 - z^2*u^3 + z*u^7 + 2*z*u^6 + z*u^5,
    y^2*u^3 - 2*z^2*u^7 - 4*z^2*u^6 - 2*z^2*u^5,
    y*z^2*u + y*z*u^3 + y*z*u^2 + z^3*u^3 + z^3*u^2 + 5*z^2*u^7 + 9*z^2*u^6 +
        4*z^2*u^5 + z^2*u^4 + z^2*u^3,
    y*z*u^4 + y*z*u^3 - 3/2*z^2*u^7 - 2*z^2*u^6 + 1/2*z^2*u^5 + z^2*u^4,
    y*u^6 + y*u^5 + z^2*u^7 + z^2*u^6 + z*u^8 + 2*z*u^7 + z*u^6,
    z^4*u^3 + z^4*u^2 - z^2*u^7 - 2*z^2*u^6 - z^2*u^5,
    z^3*u^4 + z^3*u^3,
    z^2*u^8 + 2*z^2*u^7 + z^2*u^6
]
\end{lstlisting}
\end{document}